\documentclass{amsart}
\usepackage{color,graphics}
\usepackage[matrix,arrow]{xy}
\usepackage{euscript}

\theoremstyle{plain}
\newtheorem*{lemma}{Lemma}
\newtheorem*{remark}{Remark}

\newcommand{\R}{\mathbb{R}}
\newcommand{\Z}{\mathbb{Z}}

\newcommand{\C}{\mathbb{C}}

\newcommand{\iso}{\cong}

\newcommand{\K}{\mathbb{K}}

\begin{document}
\title[Distinct symplectic structures]{Simple examples of distinct\\ Liouville type symplectic structures}
\date{revised version; December 13, 2010}
\author{Paul Seidel}
\maketitle

A symplectic form $\omega$ on an open manifold $M$ is said to be of Liouville type if the following holds. One can write $\omega = d\theta$, and the dual Liouville vector field $Z$, defined by $i_Z\omega = \theta$, can be integrated for all times. Moreover, $M$ should admit an exhausting function $h$ such that $dh(Z) > 0$ outside a compact subset. (Strictly speaking, this should be called a complete finite type Liouville structure, but we omit the adjectives for the sake of brevity.) Outside a compact subset, any such $M$ is the symplectization of a closed contact manifold, called the boundary at infinity.

It is by now well-known that a given differentiable manifold can support several such symplectic structures which are distinct, but not distinguished by classical homotopical invariants. For constructions in dimensions greater than four, including the case of Euclidean space, see \cite{seidel-smith04b, mclean09, maydanskiy09b, maydanskiy-seidel09, bourgeois-ekholm-eliashberg09, abouzaid-seidel10}. Four-dimensional instances can be obtained by attaching handles along Chekanov's examples of distinct Legendrian knots, as proved in \cite{bourgeois-ekholm-eliashberg09}. The aim of this note is to provide some rather basic examples of the same phenomenon. Of course, the proofs that their symplectic structures are different still rely on general properties of Floer cohomology, hence can't be considered elementary, but the computations involved are at least conceptually simple. The originality of the examples is somewhat limited. Those in Section \ref{sec:4} are closely related to a special case of Honda's classification of contact structures on circle bundles \cite{honda00}. Those in Section \ref{sec:6} are slight modifications of a construction from \cite{mcduff91b}. Finally, those in Section \ref{sec:8} were inspired by \cite{mclean09,seidel-smith04b}. I still hope that a concise exposition may be useful.

This work was partially supported by NSF grant DMS-1005288. I am indebted to Mohammed Abouzaid, Mark McLean, and Ivan Smith for useful conversations.

\section{A four-dimensional example\label{sec:4}}

Let $S_1$ be a once-punctured oriented surface of genus $g > 0$, and $S_2$ a $2j+1$-punctured oriented surface of genus $g-j$, for some choice of $1 \leq j \leq g$. Choose Liouville type symplectic structures on both $S_k$. Then, equip $M_k = S_k \times \R \times S^1$ with the product of those structures and the standard one on $\R \times S^1 = T^*S^1$.

\begin{lemma}
$M_1$ is diffeomorphic to $M_2$, in a way which is compatible with the homotopy classes of almost complex structures associated to the symplectic forms.
\end{lemma}

\begin{proof}
It is well-known \cite{whitehead} that $S_1 \times \R$ is diffeomorphic to $S_2 \times \R$ (both are interiors of a genus $2g$ handlebody), hence $M_1$ is diffeomorphic to $M_2$. The tangent bundle of $M_k$ is trivial as a real oriented vector bundle. Almost complex structures on that bundle correspond to maps $M_k \rightarrow SO(4)/U(2) \iso S^2$. Since $M_k$ is homotopy equivalent to a $2$-dimensional cell complex, the only obstruction to constructing a nullhomotopy for such a map lies in $H^2(M_k;\pi_2(S^2))$, and is detected by the first Chern class. In our case, both $M_k$ carry almost complex structures with zero first Chern class.
\end{proof}

We denote by $SH^*(\cdot)$ the symplectic cohomology \cite{viterbo97a} of a Liouville type symplectic manifold. Recall that symplectic cohomology is defined as the Floer cohomology of a Hamiltonian with a suitable growth behaviour at infinity. The underlying chain complex has generators coming from the interior topology and boundary dynamics (more specifically, periodic Reeb orbits on the boundary at infinity). The outcome takes the form of a $\Z/2$-graded vector space over some fixed coefficient field $\K$. For simplicity, we will take $\K = \Z/2$ throughout. Also, denote by $SH^*(\cdot)_0$ the direct summand corresponding to nullhomologous loops.

\begin{lemma}
$SH^*(M_1)_0$ is infinite-dimensional over $\K$, while $SH^*(M_2)_0$ is finite-dimensional.
\end{lemma}

\begin{proof}
$SH^*(S_k) \iso H^*(S_k;\K) \oplus \bigoplus_{i=1}^{\infty} H^*(\partial S_k;\K)$ where, in a slight abuse of notation, $\partial S_k$ is the boundary at infinity \cite[Example 3.3]{seidel07}. By decomposing into direct summands, one finds that $SH^*(S_1)_0 \iso SH^*(S_1)$, whereas $SH^*(S_2)_0 = H^*(S_2;\K)$. Similarly, $SH^*(\R \times S^1)_0 = H^*(\R \times S^1;\K)$. Oancea's K\"unneth formula \cite{oancea04} applies to $SH^*(\cdot)_0$, and completes the argument.
\end{proof}

$SH^*(\cdot)$ and $SH^*(\cdot)_0$ are invariant under symplectomorphisms that are exact with respect to the chosen Liouville one-forms. They are therefore also invariant under general symplectomorphisms $M_1 \iso M_2$, provided that at least one of the two manifolds involved has the following property: every class in $H^1(M_k;\R)$ can be realized as the flux of a symplectic isotopy (compare the discussion in \cite[Section 2]{abouzaid-seidel10}). In our case, it is easy to find the required isotopies on either $M_k$. As a consequence, $M_1$ and $M_2$ are not symplectically isomorphic.

\begin{remark}
Symplectic cohomology has the additional structure of a commutative ring, via the pairs-of-pants product. One can show that the rings $SH^*(M_2)$ are pairwise non-isomorphic for different choices of $j$, so that these also lead to distinct symplectic structures. Here's a brief sketch of the argument. Write $SH^*(\cdot)_{nil}$ for the ideal of nilpotent elements. Then
\[
SH^*(M_2)/SH^*(M_2)_{nil} \iso \K[t_1,\dots,t_{2j+1},s,s^{-1}]/I,
\]
where $I$ is the ideal generated by $t_a t_b$ for any $a \neq b$ (the $t_a$ correspond to loops around the various boundary components of $S_2$). This is the ring of functions on an affine algebraic variety over $\K$ having $2j+1$ irreducible components.
\end{remark}

\section{A six-dimensional example\label{sec:6}}

Let $S$ be a closed oriented surface of genus $\geq 2$, and $C \rightarrow S$ its tangent circle bundle. McDuff \cite{mcduff91b} showed that $T^*S \setminus S \iso C \times \R$ carries a symplectic structure of Liouville type.
Consider the product $M_1 = (C \times \R) \times \R^2$, where $\R^2$ carries the standard symplectic structure. On the other hand, let $M_2 = T^*C$ be the cotangent bundle, again with the standard symplectic structure.

\begin{lemma}
$M_1$ and $M_2$ are diffeomorphic, and the diffeomorphism can be chosen to be compatible with the homotopy classes of almost complex structures associated to the given symplectic forms.
\end{lemma}

\begin{proof}
The first statement is clear since $C$ is an oriented three-manifold, hence has trivial tangent bundle.
For the same reason, the tangent bundle of $M_k$ is trivial, hence almost complex structures on it correspond to maps $M_k \rightarrow SO(6)/U(3) \iso \C P^3$. Since $M_k$ is homotopy equivalent to a $3$-dimensional cell complex, the only obstructions to constructing a nullhomotopy for such a map lie in $H^2(M_k;\pi_2(\C P^3))$, and are again detected by the first Chern class. In McDuff's example, the first Chern class is the same as for the restriction of the standard symplectic form on $T^*S$, hence zero. The same is true for $T^*C$ since $C$ is an oriented manifold.
\end{proof}

Note that $M_2$ contains a non-displaceable closed Lagrangian submanifold, the zero-section, whereas $M_1$ doesn't. Hence, the two manifolds are not symplectically isomorphic.

\section{An eight-dimensional example\label{sec:8}}

Take a nontrivial fibered knot $K \subset S^3$. The product $E = \R \times (S^3 \setminus K)$ carries a symplectic structure of Liouville type. In fact, there is a symplectic fibration $E \rightarrow \R \times S^1$ whose fibre is the interior of the Seifert surface of $K$.

\begin{lemma}
$SH^*(E) \neq 0$.
\end{lemma}

\begin{proof}
Let $T \subset (S^3 \setminus K)$ be the boundary of a tubular neighbourhood of $K$. Since our knot is nontrivial, the map $\pi_1(T) \rightarrow \pi_1(S^3 \setminus K)$ is injective (by Dehn's Lemma). One can arrange easily that $\{0\} \times T$ is a Lagrangian torus in $E$. There can be no pseudo-holomorphic discs with boundary on that torus, for any almost complex structure compatible with the given symplectic form. By an argument outlined in \cite[Section 5]{seidel07}, the existence of such a torus implies that the image of the identity under the canonical map $H^*(E;\K) \rightarrow SH^*(E)$ is nonzero.
\end{proof}

Take a meridian of $K$. We can think of it as lying inside the boundary at infinity $\partial E$. Choose a Legendrian knot in the same free homotopy class, which always exists by the h-principle, and attach a Weinstein handle \cite{weinstein91} to it. This yields another symplectic manifold of Liouville type, denoted by $X$. Finally, set $M = X \times X$.

\begin{lemma}
$M$ is diffeomorphic to $\R^8$.
\end{lemma}

\begin{proof}
Since the meridian normally generates $\pi_1(S^3 \setminus K)$, $X$ is simply-connected. On the other hand, it is homotopy equivalent to a two-dimensional cell complex, and has Euler characteristic $1$. Hence, $X$ is contractible, and so is $M$. Moreover, $M$ is simply-connected at infinity, so the h-cobordism theorem concludes the argument.
\end{proof}

\begin{lemma}
$SH^*(M) \neq 0$.
\end{lemma}

\begin{proof}
Since $X$ is constructed by handle-attachment from $E$, there is an exact symplectic embedding $E \hookrightarrow X$. We therefore have the Viterbo functoriality \cite{viterbo97a} map, which fits into a commutative diagram
\[
\xymatrix{
H^*(X;\K) \ar[d] \ar[r] & \ar[d] SH^*(X) \\
H^*(E;\K) \ar[r] & SH^*(E).
}
\]
Since the restriction map in the left hand column maps the identity to the identity, our previous observation implies that $SH^*(X) \neq 0$. Again, the K\"unneth formula \cite{oancea04} allows us to carry over the result to $M$.
\end{proof}

Therefore, $M$ is not symplectically isomorphic to standard $\R^8$.


\end{document}